\documentclass[11pt]{article}

\usepackage{graphicx}
\usepackage{amsmath}
\usepackage{amsthm}
\usepackage{amsfonts}
\usepackage{amssymb, amscd}
\usepackage[all, knot]{xy}
\newtheorem{thm}{Theorem}[section]
\newtheorem{cor}[thm]{Corollary}
\newtheorem{lem}[thm]{Lemma}
\newtheorem{prop}[thm]{Proposition}
\newtheorem{example}[thm]{Example}
\theoremstyle{definition}
\newtheorem{defn}[thm]{Definition}
\theoremstyle{remark}
\newtheorem{rem}[thm]{Remark}
\numberwithin{equation}{section}
\input{epsf.sty}

\def\b{{\mathfrak b}}

\newcommand{\A}{\mathbb{A}}

\newcommand{\M}{M}
\newcommand{\C}{\mathbb{C}}

\newcommand{\R}{\mathbb{R}}

\newcommand{\F}{\mathbb{F}}
\newcommand{\K}{\mathbb{K}}
\newcommand{\B}{\mathbb{B}}

\newcommand{\LX}{ {\mathcal{ L}} }
\newcommand{\Hom}{\mbox{Hom}}

\begin{document}

\title{ Cohomology and Versal Deformations of Hom-Leibniz algebras}

\author{Faouzi  Ammar, Zeyneb Ejbehi and Abdenacer Makhlouf }

\maketitle

         
%
%
%
%

\begin{abstract}
The purpose of this paper is to study global deformations of Hom-Leibniz algebras. We introduce a cohomology for Hom-Leibniz algebras with values in a Hom-module, characterize versal deformations and provide an example.
\end{abstract}
MSC classification:  17A32,17A30,17B56,13D10\\
Keywords : Hom-Leibniz algebra, cohomology,  deformation, versal deformation.

  \section*{Introduction}
 In this paper we generalize the study of  versal deformations of Leibniz algebras developed in \cite{Fial,Fialowoski,Ashis} to Hom-Leibniz algebras. Hom-structures  were introduced in \cite{Ammar,F.Ammar,Makhl Silv Hom,Makhl Silv Not,Yau homol}. Cohomologies of Hom-Lie and Hom-associative algebras  
were  developed in \cite{Ammar,Makhl Silv Not}. 
The structure of the paper is as follows: in Section 1, we summarize  the definitions and   introduce a  cohomology for Hom-Leibniz algebras.
In Section 2, we introduce definitions of global deformation,  infinitesimal deformations and  versal deformation of Hom-Leibniz algebras, as well as the notions of equivalence between two global deformations.
In Section 3, we  discuss universal    infinitesimal deformations of Hom-Leibniz algebras. We construct a canonical unique infinitesimal deformation which induces all the others.
In Section 4, we recall Harrison cohomology related to commutative associative algebra and we compute obstructions.
In Section 5, we extend  universal deformations to versal one.
In section 6, we connect obstructions to Massey brackets.
We provide, in Section 7,  an explicit class of  Hom-Leibniz algebras examples for which we  calculate cohomology and versal deformations.
 \\Throughout this paper $\K$ denotes an algebraically
closed field of characteristic 0.

\section {Hom-Leibniz algebras and  Cohomology}
A class of quasi Leibniz algebras was introduced in \cite{LS1} in
connection to general quasi-Lie algebras following the standard
Loday's conventions for Leibniz algebras  \cite{Loday} (i.e. right Loday
algebras). Hom-Leibniz algebras form a subclass bordering  Hom-Lie algebras, which were discussed in   \cite{Makhl Silv Hom}.   In this section we summarize the definitions and introduce a cohomology for this class of Hom-algebras. 
\begin{defn}
A Hom-Leibniz algebra is a $\K$-module $\LX
$ equipped with a bracket operation
 and a  linear map $\alpha: \LX
\rightarrow\LX
$ that satisfy the Hom-Jacobi identity 
\begin{equation}\label{HomJacobi}
\big[\alpha(x),[y,z]\big]=\big[[x,y],\alpha(z)\big]-\big[[x,z],\alpha(y)\big]\;\;  \forall x,y,z\in \LX.
\end{equation}
The triple $(\LX, [., .], \alpha)$ denotes the Hom-Leibniz algebra. In the sequel, we deal with multiplicative Hom-Leibniz algebras, i.e. $\alpha$ is an algebra  morphism, that is the condition $\alpha [x,y]=[\alpha (x),\alpha (y)]$ is satisfied for all $x,y\in \LX$.

Let $(\LX, [., .], \alpha)$ and $(\LX^{\prime}, [., .]^{\prime}, \alpha^{\prime})$ be two Hom-Leibniz algebras. A linear map $f\ : \LX \rightarrow \LX^{\prime}$ is a \emph{morphism of Hom-Leibniz algebras} if 
$
[., .]^{\prime} \circ (f \times f) = f \circ [., .]$   and $ f \circ \alpha = \alpha^{\prime}\circ f.
$ 
It is said to be a \emph{weak morphism} if holds only the first condition.

\end{defn}
\begin{rem}
A Hom-Lie algebra is a Hom-Leibniz algebra for which the bracket is skew-symmetric.
\end{rem}
  
\begin{defn}
A representation  $\M$   of a Hom-Leibniz algebra $(\LX, [., .], \alpha)$ with
respect to  $A \in \mathfrak{gl}(\M )$ where $\M$ is 
a $\K$-module,  is defined with  two actions (left and right) on $\LX
$. These actions are denoted by the following  brackets as well 
$$[.,.]: \LX
 \times\M\longrightarrow\M\;\; \text{and} \;\; [.,.] :\M\times\LX
\longrightarrow\M\;$$
satisfying 
$$[\alpha (l),\beta (m)]=\beta [l,m]\ \text{ and } [\beta (m),\alpha (l)]=\beta [m,l]\ \text{ for } l\in \LX,\ m\in M,
$$
and such that 
$$\big[\gamma(x),[y,z]\big]=\big[[x,y], \gamma(z)\big]- \big[[x,z], \gamma(y)\big]$$
holds, whenever one of the variables is in $\M$ and the two others in $\LX
$ and 
where $\gamma =\alpha$ if the element is in $\LX
$  and $\gamma =A$ if it is in $\M$.
\end{defn}

Let $(\LX, [., .], \alpha)$  be a Hom-Leibniz algebra and $(\M,\gamma)$ a representation.

Define $\mathcal{C}^n(\LX
,\M):=Hom_{\K}(\LX
^{\times n},\M)$,$\ n\geq 0$, 
 such that a cochain $\varphi\in \mathcal{C}^{n}(\LX
,\M
)$ is an $n$-linear  map
$\varphi : \LX
^n \rightarrow \M
 \; \hbox{satisfying }\;\forall x_0,x_1,...,x_{n-1} \in \LX
$:
 $$\gamma \circ \varphi(x_0,...,x_{n-1})=\varphi\big(\alpha (x_0),\alpha(x_1),...,\alpha(x_{n-1})\big).$$

Let $\delta^n:\mathcal{C}^n(\LX
,\M)\rightarrow \mathcal{C}^{n+1}(\LX
,\M)$
be a $\K$-homomorphism defined by 
\begin{align*}
 &\delta^n\varphi(x_1,...,x_{n+1})
  =[\alpha^{n-1}(x_1),\varphi(x_2,...,x_{n+1})]+\sum_{i=2}^{n+1}(-1)^i\big[\varphi(x_1,...,
\widehat{x_i},...,x_{n+1}),\alpha^{n-1}(x_i)\big]\\
&+\sum_{1\leq i<j\leq n+1}(-1)^{j+1} \varphi(\alpha(x_1),...,\alpha(x_{i-1}),[x_i,x_j],
\alpha(x_{i+1}),...,\widehat{x_j},...,\alpha(x_{n+1})),
\end{align*}
where $\widehat{x_i}$ means that the element $x_i$ is omitted.
\begin{prop} The composite 
$\delta^{n+1}\circ\delta^n$ vanishes.

Therefore $(\mathcal{C}^{\star}(\LX
,\M),\delta )$ is a cochain complex, defining a cohomology of a Hom-Leibniz algebra $\LX$ with coefficients in the representation $\M$.

\end{prop}
\begin{proof}
The proof is obtained by straightforward computation. It is similar to proof  for  Hom-Lie algebras (see \cite{Ammar}).
\end{proof}
 The $n$-$th $ corresponding cohomology groups  are denoted by $H^{n}(\LX ,\M)=Z^{n}(\LX ,\M)/B^{n}(\LX ,\M)$, with $Z^{n}(\LX ,\M)$ being the $n$-cocycles and $B^{n}(\LX ,\M)$ the $n$-coboundaries.
 
 Similar complex for Hom-Lie algebras was defined independently in \cite{Ammar,Sheng}.
 
\section{ Global Deformations }
Let $(\LX ,[.,.],\alpha)$ be a Hom-Leibniz algebra over $\K$ and $\A$ be a commutative $\K$-algebra with a unit  $1$.
Let $\varepsilon: \A\rightarrow\K$ be a fixed  augmentation, that is an algebra homomorphism
with $\varepsilon(1)=1$. We   set  $\mathfrak{m}=ker(\varepsilon)$ and  assume that $dim({\mathfrak{m}^k}/{\mathfrak{m}^{k+1}})<\infty $ for all $k$, in order to avoid transfinite induction.
\begin{defn}
A global deformation $\lambda$ of $\LX
$ with base $(\A,\mathfrak{m})$, or simply with base $\A$, is a Hom-Leibniz
$\A$-algebra structure on the tensor product $\A\otimes_\K\LX
$ with the bracket $[.,.]_\lambda$
and a linear map  $id\otimes\alpha$  such that
$$\varepsilon\otimes id : \A\otimes\LX
\longrightarrow\K\otimes\LX
$$ is a Hom-Leibniz
$\A$-homomorphism.

In the sequel a global deformation is simply called deformation.

A deformation $\lambda$ is called \emph{infinitesimal} or of  first order (resp. \emph{order} $k$)  if in addition $\mathfrak{m}^2=0$ (resp. $\mathfrak{m}^{k+1}=0$).
\end{defn}

Observe that, by $\A$-linearity of $[.,.]_\lambda$,  we have for $l_1,l_2\in\LX
$ and $a,b\in \A$ 
$$[a\otimes l_1,b\otimes l_2]_\lambda=a b[1\otimes l_1,1\otimes l_2]_\lambda.$$
Thus, to define a deformation $\lambda$ it is enough to specify
the brackets $[1\otimes l_1,1\otimes l_2]_\lambda$ for $l_1,l_2\in\LX
$. Moreover, since
$\varepsilon\otimes id:\A\otimes\LX
\longrightarrow\K\otimes\LX
$ is a Hom-Leibniz $\A$-algebra homomorphism then 
$$(\varepsilon\otimes id)[1\otimes l_1,1\otimes l_2]_\lambda=[l_1,l_2]=(\varepsilon\otimes id)(1\otimes[l_1,l_2]).$$
Hence, we can write $$[1\otimes l_1,1\otimes l_2]_\lambda= 1\otimes [l_1,l_2]+\sum_{j}c_j\otimes y_j,$$
where $\sum_{j}c_j\otimes y_j$ is a finite sum with $c_j \in Ker(\varepsilon)=\mathfrak{m}$ and $y_j \in  \LX
.$
\begin{rem}
A deformation $\lambda$ of a Hom-Leibniz algebra $(\LX ,[.,.] , \alpha) $ with base $(\A,\mathfrak{m})$, is defined on the tensor product $\A\otimes_\K\LX
$ with a bracket  $[.,.]_\lambda$
  and a linear map  $id\otimes\alpha$  such that
\begin{enumerate}
\item for $l_1,l_2\in\LX
$ and $a,b\in \A$ we have  $[a\otimes l_1,b\otimes l_2]_\lambda=a b[1\otimes l_1,1\otimes l_2]_\lambda$,
\item the bracket $[.,.]_\lambda$ and the linear map $id\otimes\alpha$ satisfy the Hom-Jacobi identity \eqref{HomJacobi},
\item the augmentation $\varepsilon$ corresponding to   $\mathfrak{m}$ satisfies $(\varepsilon\otimes id)[1\otimes l_1,1\otimes l_2]_\lambda=1\otimes[l_1,l_2]=[l_1,l_2].$
\end{enumerate}
\end{rem}

A deformation with base $\A$ is called \emph{local} if the algebra $\A$ is local. The maximal ideal $\mathfrak{m}$  is unique in this case.  The algebra $\A$ is complete if $\A=\overleftarrow{lim}_{n\rightarrow \infty}({\A}/{\mathfrak{m}^n})$. Formal deformations are deformations with a complete local algebra as base.\\
A \emph{formal} deformation of a Hom-Leibniz algebra $\LX$, with base a complete local algebra $\A$, is  a Hom-Leibniz $\A$-algebra structure on the completed tensor product 
$\A\widehat{\otimes}\LX=\overleftarrow{lim}_{n\rightarrow \infty}({\A}/{\mathfrak{m}^n}\otimes \LX)$, which is a projective limit of deformations with base ${\A}/{\mathfrak{m}^n}$, such that $\varepsilon\widehat{\otimes}id:\A\widehat{\otimes}\LX\rightarrow \K\otimes\LX=\LX$ is a Hom-Leibniz algebras $\A$-homomorphism.

\begin{rem}
We recover  one-parameter formal deformation, introduced by Gerstenhaber \cite{Gerst def}, if $\A =\mathbb{K} [[t]]$.
\end{rem}

\begin{defn}
Suppose $\lambda_1\; \hbox{and}\; \lambda_2$ are two deformations of a Hom-Leibniz algebra
$\LX
$ with base $\A$. We call them \emph{equivalent} if there exists a Hom-Leibniz $\A$-algebra isomorphism
$$\phi:(\A\otimes \LX
,[.,.]_{\lambda_1},id\otimes\alpha) \longrightarrow(\A\otimes \LX
,[.,.]_{\lambda_2},id\otimes\alpha)$$
such that $(\varepsilon\otimes id)\circ \phi=\varepsilon\otimes id$.
\end{defn}
\begin{defn}
Suppose $\lambda$ is a given deformation of a Hom-Leibniz algebra $(\LX
,[.,.],\alpha)$ with base $(\A,\mathfrak{m})$ and augmentation
$\varepsilon: \A\longrightarrow \K$. Let $\A'$ be another commutative algebra with identity
and a fixed augmentation  $\varepsilon': \A'\longrightarrow\K$.
Suppose $\phi: \A\longrightarrow\A'$ is an algebra homomorphism with $\phi(1)=1$
and $\varepsilon'\circ\phi=\varepsilon$. 
Then the push-out $\phi_\ast\lambda$ is the deformation of $(\LX
,[.,.],\alpha)$ with base $(\A',\mathfrak{m}')$, where  $\mathfrak{m}'=Ker(\varepsilon')$,  and a bracket
$$[a_1'\otimes_{\A}(a_1\otimes l_1),a_2'\otimes_{\A}(a_2\otimes l_2)]_{\phi_\ast\lambda}=
a_1'a_2'\otimes_{\A}[a_1\otimes l_1,a_2\otimes l_2]_{\lambda} $$
where $a_1', a_2'\in \A', a_1, a_2\in \A$ and $l_1,l_2\in \LX
$. Here $\A'$ is considered as an
$\A'$-module by the map $a'\cdot a=a'\phi(a)$, so that
$$ \A'\otimes \LX
=(\A'\otimes_{\A}\A)\otimes \LX
=\A'\otimes_{\A}(\A\otimes\LX
).$$
\end{defn}
It is easy to see that $(\A'\otimes\LX
,[.,.]_{\phi_\ast\lambda},id\otimes\alpha)$ is a Hom-Leibniz algebra.
\begin{rem}
If the bracket $[.,.]_\lambda$ is given by
$$[1\otimes l_1,1\otimes l_2]_\lambda=1\otimes[l_1,l_2]+\sum_j c_j\otimes y_j\;
 \hbox{for}\; c_j\in \mathfrak{m }\;\hbox{and}\;y_j \in \LX,
$$
then the bracket $[.,.]_{\phi_\ast\lambda}$ can be written as
$$[1\otimes l_1,1\otimes l_2]_{\phi_\ast\lambda}=1\otimes [l_1,l_2]+\sum_j\phi(c_j)\otimes y_j.$$
\end{rem}
\section{Universal Infinitesimal Deformation}
In this section we construct a canonical infinitesimal deformation of a Hom-Leibniz algebra $\LX
$. It turns out that it is a universal infinitesimal deformation. We follow  to this end the procedure developed by Fialowski and her collaborators in different situations. We generalize  to Hom-Leibniz algebras, the classical result for Leibniz algebras obtained  in \cite{Fialowoski}.

Let $(\LX
,[.,.],\alpha)$ be a Hom-Leibniz algebra that satisfies the condition $ dim(H^2(\LX
,\LX
))< \infty$.
This is true for example, if $\LX
$ is finite-dimensional.
 Throughout this paper, we  denote the space $H^{2}(\LX
,\LX
)$ by $\mathbb{H}$ and its dual by $\mathbb{H}'$.
Consider the algebra $C_1=\K\oplus \mathbb{H}'$ by setting $$(k_1,h_1)\cdot (k_2,h_2)=(k_1k_2,k_1h_2+k_2h_1).$$ Observe that $\mathbb{H}'$ is an ideal of $C_1$ and $\mathbb{H}'^2=0$. Let $\mu $ be a map that takes a cohomology
class into a representative cocycle,
$$ \mu:\mathbb{H}\longrightarrow \mathcal{C}^2(\LX
,\LX
)=Hom(\LX
^2,\LX
).$$  Notice that there is an isomorphism
$\mathbb{H}'\otimes \LX
 \cong Hom(\mathbb{H},\LX
)$, so we have $$C_1\otimes\LX
 =\LX
\oplus Hom(\mathbb{H},\LX
).$$
Using the above identification, define a Hom-Leibniz algebra structure  on $C_1\otimes\LX
$ as follows.

  For
$(l_1,\phi_1),(l_2,\phi_2) \in \LX
\oplus Hom(\mathbb{H},\LX
)$,  let
$$\big[(l_1,\phi_1),(l_2,\phi_2)\big]=([l_1,l_2],\psi) $$
where the map $\psi:\mathbb{H}\longrightarrow\LX
$ is defined as 
$$\psi(f)=\mu(f)(l_1,l_2)+[\phi_1(f),l_2]+[l_1,\phi_2(f)]\;\hbox{for}\;f\in \mathbb{H}.$$
Define a  linear map $\alpha\otimes \beta$ on $C_1\otimes\LX$ such that 
\begin{eqnarray*}
\beta:& Hom(\mathbb{H},\LX)\longrightarrow Hom(\mathbb{H},\LX )\\
\ & \varphi\longrightarrow\beta(\varphi)
\end{eqnarray*}
where $\beta(\varphi): f\rightarrow\alpha(\varphi(f))$ for $f\in \mathbb{H}.$
\begin{prop}
The triple $(\LX
\oplus Hom(\mathbb{H},\LX
), [. , .], \alpha\otimes \beta)$ is a Hom-Leibniz algebra.
\end{prop}
\begin{proof}
 Let $l_1,l_2,l_3\in \LX
 \;\hbox{and}\; \phi_1,\phi_2,\phi_3 \in Hom(\mathbb{H},\LX
)$
$$\big[(\alpha(l_1),\beta(\phi_1)),[(l_2,\phi_2),(l_3,\phi_3)]\big]-\big[[(l_1,\phi_1),(l_2,\phi_2)],
(\alpha(l_3),\beta(\phi_3))\big]$$ $$+\big[[(l_1,\phi_1),(l_3,\phi_3)],(\alpha(l_2),\beta(\phi_2))\big]$$
 $$=\big[(\alpha(l_1),\beta(\phi_1)),([l_2,l_3],\psi_{2,3})\big]- \big[([l_1,l_2],\psi_{1,2}),(\alpha(l_3),\beta(\phi_3))\big]$$ $$+\big[([l_1,l_3],\psi_{1,3}),(\alpha(l_2),\beta(\phi_2))\big]$$
 $$=\big(\big[\alpha(l_1),[l_2,l_3]\big],\psi_1 \big)-\big(\big[[l_1,l_2],\alpha(l_3)\big],\psi_2\big)
 +\big(\big[[l_1,l_3],\alpha(l_2)\big],\psi_3\big) $$
$$=\big(\big[\alpha(l_1),[l_2,l_3]\big]-\big[[l_1,l_2],\alpha(l_3)\big]+
 \big[[l_1,l_3],\alpha(l_2)\big],\psi_1-\psi_2+\psi_3\big).$$
The first coordinate is the left hand side of  Hom-Leibniz identity, which vanishes. Besides  the second coordinates vanishes as well since  $$\psi_1-\psi_2+\psi_3=\mu(f)(\alpha(l_1),[l_2,l_3])
-\mu(f)([l_1,l_2],\alpha(l_3))+\mu(f)([l_1,l_3],\alpha(l_2))$$
$$+ \underline{[\beta(\phi_1)f,[l_2,l_3]]}-\underline{\underline{\underline{[[l_1,l_2],\beta(\phi_3)f]}}}
+\underline{\underline{[[l_1,l_3],\beta(\phi_2)f]}}$$ $$+[\alpha(l_1),\mu(f)(l_2,l_3)]
+\underline{\underline{[\alpha(l_1),[\phi_2(f),l_3]}}
+\underline{\underline{\underline{[\alpha(l_1),[l_2,\phi_3(f)]}}}$$
$$+ [\mu(f)(l_1,l_3),\alpha(l_2)]+\underline{[[\phi_1(f),l_3],\alpha(l_2)]}
+\underline{\underline{\underline{[[l_1,\phi_3(f)],\alpha(l_2)]}}}$$
$$- [\mu(f)(l_1,l_2),\alpha(l_3)]-\underline{[[\phi_1(f),l_2],\alpha(l_3)]}
-\underline{\underline{[[l_1,\phi_2(f)],\alpha(l_3)]}}$$ $$=\delta^2\mu(f)=0.$$
\end{proof}
The triple $(\LX
\oplus Hom(\mathbb{H},\LX
), [. , .], \alpha\otimes \beta)$ defines an infinitesimal  deformation of a Hom-Leibniz algebra $\LX$ which we denote by $\eta_1.$

The main property of $\eta_1$ is that it is universal in the class of infinitesimal deformations.

 \begin{prop}
 Up to isomorphism, the deformation $\eta_1$ does not depend on the choice of $\mu$.
 \end{prop}
 \begin{proof}
 Let $\mu'$ be another choice for $\mu$, $$\mu':H\longrightarrow \mathcal{C} ^2(\LX,\LX) .$$
 Then for $h\in \mathbb{H},\mu(h)\;\hbox{and}\;\mu'(h)$ represent the same class.
 We can define a homomorphism $$f:\mathbb{H}\longrightarrow  \mathcal{C} ^1(\LX,\LX) $$
 by $f(h_i)=h_i$, where $\{h_i\}_i$ is a basis of $\mathbb{H}$, with $\delta f(h_i)=\mu(h_i)-\mu'(h_i)$.
\\    By the identification $C_1\otimes \LX
\cong\LX
\oplus Hom(\mathbb{H},\LX
)$,
$$ \rho :\LX
\oplus Hom(\mathbb{H},\LX
)\longrightarrow\LX
\oplus Hom(\mathbb{H},\LX
)\;\hbox{by} \; \rho(l,\phi)=(l,\psi)$$
where $\psi(h)=f(h)l+\phi(h),l\in \l,\phi\in Hom(\mathbb{H},\LX
)$, we have 
\begin{itemize}
  \item $\rho$ is $C_1$-(linear) automorphism of $C_1\otimes\LX
$ with $\rho^{-1}(l,\psi)=(l,\phi)$
where $\phi(h)=\psi(h)-f(h)l$

  \item $\rho$ preserves the bracket. Indeed,
  let $(l_1,\phi_1)$ and $(l_2,\phi_2)\in C_1\otimes\LX
$ with $\rho(l_i,\phi_i)=(l_i,\psi_i); i=1,2$

$[(l_1,\phi_1),(l_2,\phi_2)]=([l_1,l_2],\phi_3)$ where $\phi_3(h)=\mu(h)(l_1,l_2)+[\phi_1(h),l_2]+[l_1\phi_2(h)]$

 and $[(l_1,\psi_1),(l_2,\psi_2)]=([l_1,l_2],\psi_3)$ where 
 \begin{eqnarray*}
 \psi_3(h)&=&\mu(h)(l_1,l_2)+[\psi_1(h),l_2]+[l_1\psi_2(h)]\\
  &=& \mu(h)(l_1,l_2)-\delta f(h)(l_1,l_2)+[\phi_1(h)+f(h)l_1,l_2]+[l_1, \phi_2(h)+f(h)l_2]\\
&=&\mu(h)([l_1,l_2])-\underline{[l_1,f(h)l_2]}-\underline{\underline{[f(h)l_1,l_2]}}+f(h)(l_1,l_2)+\\ & &[\phi_1(h),l_2]+\underline{\underline{[f(h)l_1,l_2]}}+   [l_1,\phi_2(h)]+\underline{[l_1,f(h)l_2]}\\
& =&\mu(h)([l_1,l_2])+f(h)(l_1,l_2)+[\phi_1(h),l_2]+   [l_1,\phi_2(h)]\\
&=& \phi_3(f)+f(h)([l_1,l_2]). \end{eqnarray*}
\end{itemize}
Hence $$\rho[(l_1,\phi_1),(l_2,\phi_2)]=[\rho(l_1,\phi_1),\rho(l_2,\phi_2)].$$
Therefore, up to an isomorphism, the infinitesimal deformation obtained is independent
of the choice of $\mu$.
\end{proof}
 \begin{rem}
We have $dim(\mathbb{H})<+\infty$. Suppose that $\{h_i\}_{1\leq i\leq r}$ is a basis of $\mathbb{H}$
and $\{g_i\}_{1\leq i\leq r}$ is the dual basis. Let $\mu(h_i)=\mu_i\in  \mathcal{C} ^2(\LX
,\LX
)$. By identification
$C_1\otimes\LX
 =\LX
\oplus Hom(\mathbb{H},\LX
)$, an element $(l,\varphi)\in  \LX
\oplus\Hom(\mathbb{H},\LX
)$
corresponds to $1\otimes l+\sum_{i=1}^rg_i\otimes \varphi(h_i)$. In particular,
  $g\otimes l=\sum_{i=1}^rg_i\otimes g(h_i)l\in C_1\otimes\LX
$ corresponds to $(0,\varphi)$. For
 $f\in \mathbb{H}$ we have $\varphi(f)=g(f)l$.
 Then for $(l_1,\varphi_1),(l_2,\varphi_2) \in \LX
\oplus\Hom(\mathbb{H},\LX
)$
their bracket $([l_1,l_2],\psi)$ corresponds to
$$1\otimes[l_1,l_2]+\sum_{i=1}^rg_i\otimes(\mu_i(l_1,l_2)+[\varphi_1(h_i),l_2]+[l_1,\varphi_2(h_i)]).$$
In particular, for $l_1,l_2\in \LX
$ we have
$$[1\otimes l_1,1\otimes l_2]_{\eta_{1}}= 1\otimes [l_1,l_2]+\sum_{i=1}^rg_i\otimes \mu_{i}(l_1,l_2).$$
Hence $(\LX
\oplus Hom(\mathbb{H},\LX
),[.,.],\alpha\otimes\beta)$ is a Hom-Leibniz algebra, is equivalent to
$(C_1\otimes \LX
,[.,.]_{\eta_{1}},id\otimes\alpha)$ is a Hom-Leibniz algebra too.
\end{rem}
\begin{prop}
For any infinitesimal deformation $\lambda$ of a Hom-Leibniz algebra $(\LX
,[.,.],\alpha)$ with
a finite dimensional base $\A$,  there exists a unique homomorphism $\phi=id+a_\lambda :C_1=\K\oplus \mathbb{H}'\longrightarrow\A$
such that $\lambda$ is equivalent to the push-out $\phi_\ast \eta_1$.
\end{prop}
\begin{lem}\label{lem}
Let $\lambda$  be an infinitesimal deformation of the Hom-Leibniz algebra $(\LX
,[.,.],\alpha)$  with a finite dimensional base $\A$.
Let $m_{1\leq i\leq r}$ be a basis of $\mathfrak{m}=ker(\varepsilon)$ and $\{\xi_i\}_{1\leq i\leq r}$ be the dual basis.
Note that any element $\xi$ of $\mathfrak{m}'$  can be viewed as an element in the dual space $\A'$ with $\xi(1)=0$.
Set, for any $\xi$, $$\psi_{\lambda,\xi}(l_1,l_2)=\xi\otimes id([1\otimes l_1,1\otimes l_2]_\lambda)\;\hbox{for}\; l_1,l_2\in \LX
.$$
Then, $\psi_{\lambda,\xi}$ is a 2-cocycle.
\end{lem}
\begin{proof}
If we set $\psi_i=\psi_{\lambda,\xi_i} \;\hbox{for}\;1\leq i\leq r$,  the Hom-Leibniz bracket over $\A\otimes\LX
$
takes the form
 $$ [1\otimes l_1,1\otimes l_2]_\lambda=1\otimes[l_1,l_2]+\sum_{i=1}^r m_i\otimes x_i=1\otimes[l_1,l_2]+\sum_{i=1}^r m_i\otimes\psi_i(l_1,l_2).$$
 We have
\begin{eqnarray*}
\delta^2\psi_{\lambda,\xi}(l_1,l_2,l_3)&=&\big[\alpha(l_1),\psi_{\lambda,\xi}(l_2,l_3)\big]+\big[\psi_{\lambda,\xi}(l_1,l_3),\alpha(l_2)\big]-\big[\psi_{\lambda,\xi}(l_1,l_2),\alpha(l_3)\big]\\
&& -\psi_{\lambda,\xi}([l_1,l_2],\alpha(l_3))+\psi_{\lambda,\xi}([l_1,l_3],\alpha(l_2))+\psi_{\lambda,\xi}(\alpha(l_1),[l_2,l_3]).
\end{eqnarray*}
Observe that 
\begin{eqnarray*}
&& (\xi\otimes id)(\big[1\otimes \alpha(l_1),[1\otimes l_2,1\otimes l_3]\big]_\lambda
\\ &&=(\xi\otimes id)(\big[1\otimes \alpha(l_1),1\otimes [l_2,l_3]\big]_\lambda+\big[1\otimes\alpha(l_1),\sum_{j=1}^r m_j\otimes x_j\big]_\lambda\\ &&=\psi_{\lambda,\xi}(\alpha(l_1),[l_2,l_3])+\sum_{j=1}^r(\xi\otimes id)[1\otimes \alpha(l_1),m_j\otimes x_j]_\lambda.
\end{eqnarray*}
Moreover, 
\begin{eqnarray*}
(\xi\otimes id)[1\otimes l_i,1\otimes l_j]_\lambda&=&(\xi\otimes id)m_j[1\otimes l_i,1\otimes l_j]_\lambda\\ &=& (\xi\otimes id)m_j(1\otimes [l_1,x_j]+\sum_{i=1}^rm_i\otimes x_{ji})\\
& =&(\xi\otimes id)m_j(1\otimes [l_1,x_j]) (m^2=0)\\
&=&[l_1,(\xi\otimes id)(m_j\otimes x_j)] .
\end{eqnarray*}
Therefore 
\begin{eqnarray*}(\xi\otimes id)(\big[1\otimes \alpha(l_1),[1\otimes l_2,1\otimes l_3]_\lambda\big]_\lambda 
=\psi_{\lambda,\xi}(\alpha(l_1),[l_2,l_3])+[l_1,(\xi\otimes id)\sum_{j=1}^r m_j\otimes x_j]
\\=\psi_{\lambda,\xi}(\alpha(l_1),[l_2,l_3])+\big[l_1,(\xi\otimes id)\big([1\otimes l_2,1\otimes l_3]_\lambda-1\otimes [l_2,l_3]\big)\big] \\=\psi_{\lambda,\xi}(\alpha(l_1),[l_2,l_3])+\big[l_1,\psi_{\lambda,\xi}(l_2,l_3)].
\end{eqnarray*}
Similarly $$(\xi\otimes id)(\big[[1\otimes l_1,1\otimes l_2]_\lambda,1\otimes \alpha(l_3)\big]_\lambda)=\psi_{\lambda,\xi}([1\otimes l_1,1\otimes l_2],\alpha(l_3))+[\psi_{\lambda,\xi}(l_1,l_2),l_3],$$
$$(\xi\otimes id)(\big[[1\otimes l_1,1\otimes l_3]_\lambda,1\otimes \alpha(l_2)\big]_\lambda)=\psi_{\lambda,\xi}([1\otimes l_1,1\otimes l_3],\alpha(l_2))+[\psi_{\lambda,\xi}(l_1,l_3),l_2]. $$
It follows that $$\delta^2 \psi_{\lambda,\xi}(l_1,l_2,l_3)=(\xi\otimes id)\big(\big[[1\otimes l_1,1\otimes l_3]_\lambda,1\otimes \alpha(l_2)\big]_\lambda
-\big[[1\otimes l_1,1\otimes l_2]_\lambda,1\otimes \alpha(l_3)\big]_\lambda$$
$$+\big[1\otimes \alpha(l_1),[1\otimes l_2,1\otimes l_3]_\lambda\big]_\lambda\big)=0\;\hbox{(by the Hom-Leibniz identity)}.$$
\end{proof}

The proof of the proposition is the same  that for  Leibniz algebras, see \cite{Fialowoski} .

\section{Obstructions}
The aim of this section is to study obstructions in extending deformations.
For this end, we need the interpretation of the $1$- and $2$-dimensional Harrison cohomology of a commutative
algebra. For the definition and connections between Harrison and Hochschild cohomologies, see \cite{Harrison}.
\begin{defn}
For an $\A$-module $M$, we set $$H^q_{Harr}(\A,\M)=H^q(Ch(\A),\M).$$

\end{defn}
\begin{prop}
Let $\A$ be a local commutative $\K$-algebra with the maximal ideal $\mathfrak{m}$
 and  $\M$ be an $\A$-module with $ \mathfrak{m}\M=0$. Then we have the canonical isomorphisms
 $$ H_{Harr}^q(\A,\M)=H_{Harr}^q(\A,\K)\otimes\K.$$
\end{prop}
\begin{defn}
An extension $\B$ of an algebra $\A$ by an $\A$-module $\M$ is a $\K$-algebra $\B$ together
with an exact sequence of $\K$-modules  $$\begin{array}{ccc}
0\longrightarrow\M& \stackrel{i}{\longrightarrow}\B& \stackrel{p}{\longrightarrow}\A\longrightarrow 0 \end{array}$$
where $p$ is a $\K$-algebra homomorphism, and the $\B$-module structure  on $i(\M)$ is given by the $\A$-module structure of $\M$
by
 $$ i(m)\cdot b=i(m p(b)).$$
\end{defn}
\begin{prop}\  \begin{enumerate}
              \item  The space  $ H_{Harr}^1(\A,\M)$ is isomorphic  to the space of derivations $f:\A\rightarrow\M$.
              \item  Elements of $H_{Harr}^2(\A,\M)$ correspond bijectively to isomorphism classes of extensions
               $$\begin{array}{ccc}
0\longrightarrow\M& \stackrel{i}{\longrightarrow}\B& \stackrel{p}{\longrightarrow}\A\longrightarrow 0 \end{array}$$
 of the algebra $\A$ by means of $\M$.
              \item   The space  $ H_{Harr}^1(\A,\M)$ can be interpreted as the group
               of automorphisms of any given extensions of $\A$ by $\M$.
            \end{enumerate}
\end{prop}
\begin{cor}
If $\A$ is a local algebra with the maximal ideal $ \mathfrak{m}$,
then $$ H_{Harr}^1(\A,\M)\cong (\frac{\mathfrak{m}}{\mathfrak{m}^2})'=T\A.$$
\end{cor}

If $\A$ is a local algebra with the maximal ideal $\mathfrak{m}$,
then $$ H_{Harr}^1(\A,\M)\cong (\frac{\mathfrak{m}}{\mathfrak{m}^2})'=T\A.$$
Let $\lambda$ be a deformation of a Hom-Leibniz $\LX
$ with a finite dimensional local base $\A$
and an augmentation $\varepsilon$. consider $[f]\in H_{Harr}^2(\A,\K)$.
Suppose  $$\begin{array}{ccc}
0\longrightarrow\K& \stackrel{i}{\longrightarrow}\B& \stackrel{p}{\longrightarrow}\A\longrightarrow 0 \end{array}$$
 is a representative of the class of 1-dimensional extensions of $\A$,
corresponding to the cohomology class of $f$.
Let $I=i\otimes id:\K\otimes\l\cong\LX
\longrightarrow\B\otimes\LX
, P=p\otimes id:\B\otimes\LX
\longrightarrow\A\otimes\LX
\; \hbox{and}\;E=\hat{\varepsilon}\otimes id:\B\otimes\LX
\longrightarrow \K\otimes\LX
\cong\LX
$, where $\hat{\varepsilon}=\varepsilon\circ p$
 is the augmentation of $\B$ corresponding to the augmentation $\varepsilon$ of $\A$.
Fix a section $q : A \longrightarrow B$ of $p$ in the above extension, then the map 
$\B \rightarrow \A \oplus \K$ defined as
\begin{equation}\label{(4.3.1)}b \longrightarrow (p(b), i^{-1}(b - q(p(b)))
\end{equation}
is a $\K$-module isomorphism. Let us denote by $(a,k)_q \in \B$ the inverse of $(a,k)\in (\A \oplus\K)$
under the above isomorphism.
The cocycles $f$ representing the extension is determined by $f\big((a_1, 0)_q (a_2, 0)_q = (a_1a_2, 0)_q\big) $.
On the other hand $f$ determines the algebra structure of $\B$ by
\begin{equation}\label{(4.3.2)}(a_1, k_1)_q  (a_2, k_2)_q = (a_1a_2, a_1 k_2 + a_2 k_1 + f(a_1, a_2)_q ).\end{equation}
Suppose $dim(A) = r + 1$ and ${(m_i)}_{1\leq i\leq r}$ is a basis of the maximal ideal $\mathfrak{m_{\A}}$ of $\A$. Then
${(n_i)}_{1\leq i\leq r+1}$ is a basis of the maximal ideal $\mathfrak{m}_{\B} = p^{-1}(m_{\A})$ of $\B$,
where $n_j = (m_j, 0)_q$, for $1 \leq j\leq r$ and $n_{r+1} = (0, 1)_q$. Take the dual basis ${(\xi_i)}1\leq i\leq r$ of $\mathfrak{m'_{\A}}$. Then by the notations in Lemma \ref{lem}, we have 2-cochains $\psi_i=\psi_{\lambda,\xi_i} \in   \mathcal{C} ^2(\LX,\LX) \;\hbox{for}\;1\leq i\leq r$
 such that $[.,.]_\lambda$ can be written as
$$ [1\otimes l_1,1\otimes l_2]_\lambda=1\otimes[l_1,l_2]+\sum_{i=1}^r m_i\otimes \psi_i(l_1,l_2)\; \hbox{for}\; l_1,l_2\in \LX.
$$
Let $\chi\in  \mathcal{C} ^2( \LX,\LX) = Hom(\LX^2,\LX)$ be an arbitrary element. Define a $\B$-bilinear operation
$(B\otimes \LX)^2  \longrightarrow\B\otimes \LX$,
$$\{b_1\otimes l_1, b_2\otimes l_2\} = b_1 b_2
\otimes[l_1, l_2] +
\sum_{j=1}^r n_j \psi_j(l_1, l_2) + b_1b_2  n_{r+1} \otimes \chi(l_1, l_2).$$
It is easy to check that the $\B-$ bilinear map $\{.,.\}$ satisfies:
For $l_1, l_2 \in\B\otimes \LX
\;\hbox{and} \;l\in \LX
$
\begin{itemize}
    \item \begin{equation}\label{(4.3.3)}P\{l_1, l_2\} = [P(l_1), P(l_2)], \end{equation}
\item \begin{equation}\label{(4.3.4)}[I(l), l_1] = I[l,E(l_1)]. \end{equation}
\end{itemize}
So the Hom-Leibniz algebra structure $\lambda$ on $\A\otimes \LX
$ can be lifted to a $\B$-bilinear operation
$\{.,.\}: (\B\otimes \LX
)^{\otimes 2} \longrightarrow\B\otimes \LX$ satisfying \eqref{(4.3.3)} and \eqref{(4.3.4)}.
Define
$$\phi: (B\otimes \LX
)^{\otimes3} \longrightarrow (\B\otimes\LX
)$$
for  $b_1\otimes l_1, b_2\otimes l_2, b_3\otimes l_3\in \B\otimes\LX
$ by
\begin{eqnarray*}
 \phi(b_1\otimes l_1,b_2\otimes l_2, b_3\otimes l_3) &=&
\{b1\otimes\alpha(l1),\{b_2\otimes l_2,b3\otimes l_3\}\} + \{\{b_1
\otimes l_1,b_2\otimes l_2\}, b_3\otimes\alpha(l_3)\}\\ & & +\{\{b_1\otimes l_1, b_3\otimes l_3\}, b_2\otimes \alpha(l_2)\}.
\end{eqnarray*}
It is clear that $\{.,.\}$ satisfies the Hom-Leibniz relation if and only if $\phi = 0$. Now from property
\eqref{(4.3.3)} and the definition of $\phi$ it follows that $P\circ\phi(b_1
\otimes l_1, b_2\otimes l_2, b_3\otimes l_3) = 0$ ; for $b_1\otimes l_1, b_2\otimes l_2, b_3\otimes l_3) \in \B \otimes \LX$.
 There for $\phi$ takes values in $ker(P)$. Observe that $\phi(b_1\otimes l_1, b_2\otimes l_2, b_3\otimes l_3) = 0$,
whenever one of the arguments belongs to $ker(E)$. Suppose $b_1\otimes l_1 \in ker(E) \subset \B \otimes L$. Since
$ker(E) = ker(\hat{\varepsilon})\otimes \LX
 = p^{-1}(ker(\varepsilon))\otimes L = \mathfrak{m}_{ \B} \otimes\LX
$, we can write $b_1\otimes l_1 = \sum^{r+1}_{j=1} n_j \otimes l'_j$ with
$l'_j\in \LX
,  j = 1 ... r + 1$. Therefore $b_2\otimes l_2,  b_3\otimes l_3 \in \B\otimes L$, we get
$\phi(b_1\otimes l_1, b_2\otimes l_2, b_3\otimes l_3) = \phi(\sum^{r+1}_ {j=1}n_j\otimes l'_j, b_2\otimes l_2, b_3\otimes$ $l_3)=\sum^{r+1}_{j=1}n_j\phi(1\otimes l'j, b_1\otimes l_2, b_3\otimes l_3) = 0.$
This is because $\phi$ takes values in $ker(P) = im(I) = im(i)\otimes \LX = i(\K)\otimes \LX
$ and yet for any
element $k \in\K$ and $l \in\LX
$, $n_j\cdot i(k)\otimes l = i(p(n_j)k)\otimes l = i(m_j\cdot k)\otimes l = i(\varepsilon(m_j)k)\otimes l = 0$ for $ j = 1... r + 1$
 and $n_{r+1}\cdot i(k)\otimes l = k n^2_{r+1}\otimes l = 0$ where $( m_j \in \mathfrak{m}\subset \A \;\hbox{and}\;  m_j\cdot k = \varepsilon(mj)k)$.
 The other two cases are similar. Thus defines a linear map
$$\tilde{\phi} : (\frac{\B\otimes\LX
}{ker(E)})^{\otimes3}\longrightarrow ker(P).$$
Moreover, the surjective map  $E : \B\otimes\LX
 \longrightarrow\K\otimes\LX
 \cong\LX
$, defined by
$b\otimes l \longrightarrow \hat{\varepsilon}(b)\otimes l$, induces
an isomorphism $\frac{\B\otimes\LX
}{ker(E)}\stackrel{\tau}{\cong} \LX
$ , where
$$\tau: \LX \longrightarrow \frac{\B\otimes\LX
}{ker(E)};\tau (l) = 1\otimes l + ker(E).$$
Also, $ker(P) = im(I) = i(\K)\otimes \LX
 = \K i(1)\otimes \LX
\stackrel{\beta}{\cong}\LX
,$
where the isomorphism $\beta$  is given by $\beta(k n_{r+1}\otimes l)=k l$
with inverse $\beta^{-1}(l)=n_{r+1}\otimes l$. Thus we get a linear map $\bar{\phi} :\LX
^{\otimes3}\longrightarrow\LX
$, such that $\bar{\phi}=\beta\circ\phi\circ\tilde{f}^{\otimes3}$.
 The cochains $\bar{\phi}\in \mathcal{C}^3( \LX
,\LX
)$ and $\phi$ are related by
$n_{r+1}\otimes\bar{\phi}(l_1, l_2, l_3) = \phi(1\otimes l_1, 1\otimes l_2, 1\otimes l_3).$
\begin{prop}  The cochain $\bar{\phi}$ is a 3-cocycle.
\end{prop}
\begin{proof}
 The first term of $\beta^{-1} \circ\delta\bar{\phi}$ is as follows.
 \begin{eqnarray*}
\beta^{-1}([\alpha^2(l_1), \bar{\phi}(l_2, l_3, l_4)]&= &n_{r+1}\otimes[\alpha^2(l_1),\bar{\phi} (l_2, l_3, l_4)]\\
&=& I([\alpha^2(l_1),\bar{\phi}(l_2, l_3, l_4)]) (i(1) = n_{r+1})\\
&=& I([\alpha^2(l_1),E(1\otimes \bar{\phi} (l_2, l_3, l_4))])\\
&=&\{I(\alpha^2(l_1)), 1\otimes\bar{\phi} (l_2, l_3, l_4)\} \;\hbox{by \eqref{(4.3.4)}}\\
&=& \{n_{r+1}\alpha^2(l_1), 1\otimes\bar{\phi} (l_2, l_3, l_4)\}\\
&=& \{1\otimes\alpha^2(l_1), n_{r+1}\otimes\bar{\phi} (l_2, l_3, l_4)\}\\
&=& \{1\otimes\alpha^2(l_1),\bar{\phi}(1\otimes l_2, 1\otimes l_3, 1\otimes l_4)\}\\
&=& \{1\otimes \alpha^2(l_1), \{1\otimes\alpha(l_2), \{1\otimes l_3, 1\otimes l_4\}\}
\\ & &-\{1\otimes\alpha^2(l_1),\{\{1\otimes l_2, 1\otimes l_3\}, 1\alpha(l_4)\}\} \\
&&+\{1\otimes\alpha^2(l_1), \{\{1\otimes l_2, 1\otimes l_4\}, 1\otimes\alpha(l_3)\}\}.
\end{eqnarray*}
Similarly, computing other terms and substituting in the expression of $\beta^{-1} \circ\delta\bar{\phi}$, we get
$$\beta^{-1} \circ\delta\bar{\phi}(l_1, l_2, l_3,l_4) = 0 \;\hbox{i.e.}\;\; \delta^3 \bar{\phi} = 0.$$
\end{proof}
Let us show now that the cohomology class of $\bar{\phi}$ is independent of the choice of the lifting $\{.,.\}$.
Suppose $\{.,.\}$ and $\{.,.\}'$ are two $\B$-bilinear operations on $\B\otimes\LX
$, lifting the Hom-Leibniz
structure $\lambda$ on $A\otimes\LX
$. Let $\bar{\phi}$ and  $\bar{\phi}'$ be the corresponding cocycles. Set $\rho = \{.,.\}-\{.,.\}'$.
 Then
$\rho : (\B\otimes\LX
)^{\otimes2}\longrightarrow \B\otimes\LX
$ is a $\B$-linear map. Observe that
$$P \circ\rho(b_1\otimes l_1, b_2\otimes l_2) = [P(b_1\otimes l_1), P(b_2 \otimes l_2)]_\lambda - [P(b_1\otimes l_1), P(b_2\otimes
 l_2)]_\lambda = 0\;\hbox{by}\; \eqref{(4.3.3)}$$
$\rho$ takes values in $ker(P)$ and induces a linear map
$$\tilde{\rho}:( \frac{\B\otimes\LX
}{ker(E)})^{\otimes 2}\longrightarrow ker (P), $$
$$\tilde{\rho}(b_1\otimes l_1+ker(E),b_2\otimes l_2+ker(E))=\rho(b_1\otimes l_1,b_2\otimes l_2) \;\hbox{for}\b_1\otimes l_1,b_2\otimes l_2\in\B\otimes\LX
. $$
Hence we get a 2-cochain  $\bar{\rho}: \LX
^{\otimes2}\longrightarrow \LX
$ such that
$ \bar{\rho}=\beta\circ\bar{\rho}\circ\alpha^{\otimes2}\in   \mathcal{C} ^2(\LX
, \LX
)$. As before,
for $l_1, l_2\in \LX
$, we have $n_{r+1}\otimes\bar{\rho}(l_1, l_2) = \rho(1\otimes l_1, 1\otimes l_2)$.
Then a straightforward computation yields.
Hence $( \bar{\phi}'-\bar{\phi})=\delta\bar{\rho}$.
Suppose a $\B$-bilinear operation $\{.,.\}$ on $\B\otimes \LX
$ lifting the Hom-Leibniz algebra structure
$[., .]_\lambda$ on $\A\otimes\LX
$. Then any other $\B$-bilinear operation on $\B\otimes \LX
$,
lifting $[.,.]_\lambda$ is determined by a 2-cochain $\rho$ as follows.
Define $\{.,.\}':(\B\otimes\LX
)^{\otimes 2}\longrightarrow (\B\otimes\LX
)$
 by $\{1\otimes l_1,1\otimes l_2\}'=\{1\otimes l_1,1\otimes l_2\}+I\circ\rho(E(1\otimes l_1),E(1\otimes l_2))$ for $1\otimes l_1 ,1\otimes l_2\in \B\otimes \LX
.$
Then it is easy to see that $\{.,.\}'$ is lifting of $[.,.]_\lambda$ such that the cochain
 $ \bar{\rho}$ induced
by the difference  $\{.,.\}'-\{.,.\}$ is the given 2-cochain $\rho$ .
The above consideration defines a map $\theta_\lambda:H^2_{Harr}(\A,\K)\longrightarrow H^3(\LX
,\LX)$
 by $\theta_\lambda([f])=[\bar{\phi}]$ ,where $[\bar{\phi}]$ is the cohomology class of $\bar{\phi}$.
  The map $\theta_\lambda$ is called the obstruction map.
\begin{prop} Let $\lambda$ be a deformation of the Hom-Leibniz algebra  $\LX
$ with base $\A$
and  $\B$ be a $1$-dimensional extension of $\A$ corresponding to the cohomology class $[f]\in
H^2_{Harr}(\A,\K)$. Then $\lambda$ can be extended to a deformation of  $\LX
$ with base $\B$ if only if the
obstruction $\theta_\lambda([f]) = 0$.
\end{prop}
\begin{proof} Suppose $\theta_\lambda([f]) = 0$. let
\begin{equation}\label{(1)}
\begin{array}{ccc}
0\longrightarrow\K& \stackrel{i}{\longrightarrow}\B& \stackrel{p}{\longrightarrow}\A\longrightarrow 0 \end{array} \end{equation}
be a $1$-dimensional extension representing the cohomology class $[f]$. Let $\{.,.\}$ be a lifting of
the Hom-Leibniz algebra structure $\lambda$ on $\A\otimes\LX
$ to a $\B$-bilinear operation on $\B\otimes\LX
$.
 Let $\bar{\phi}$
be the associated cocycle in $ \mathcal{C} ^3( \LX
, \LX
)$ as described above. Then $\theta_\lambda([f]) =  [\bar{\phi}] = 0$ implies
$\bar{\phi} =\delta\rho $ for some $\rho\in  \mathcal{C} ^2( \LX
, \LX
)$. Now take $\rho'=-\rho$ and define a new linear map
$\{.,.\}':(\B\otimes\LX
)^{\otimes 2}\longrightarrow (\B\otimes\LX
)$
by $\{1\otimes l_1,1\otimes l_2\}'=\{1\otimes l_1,1\otimes l_2\}+I\circ\rho(E(1\otimes l_1),E(1\otimes l_2))$. 
If $\bar{\phi'}$ denotes the cocycle corresponding to $\{.,.\}'$, we have $\bar{\phi'}-\bar{\phi} = \delta\bar{\rho'}=-\bar{\phi}$.
 Hence $\bar{\phi'} = 0$
which implies $\phi' =0$. Therefore,  $\{.,.\}'$ is a Hom-Leibniz algebra structure
 on $\B\otimes\LX
$ extending $\lambda$. The converse is obvious.
\end{proof}

Let $S$ be the set of all isomorphism classes of deformation $\mu$ of $\LX$ with base $B$ such that $p_{\ast}\mu=\lambda$. The group of automorphisms $Aut$ of the extension \eqref{(1)}
has a natural action $\sigma_1$ of $\A$ on $S$, given by $\mu\rightarrow u_{\ast}\mu$ for $u\in \A$. Suppose that $\mu$ and $\mu'$ are two deformations of $\LX$ with base $B$  such that $p_{\ast}\mu=p_{\ast}\mu'=\lambda$. Let $\psi\in  \mathcal{C} ^2(\LX,\LX)$ be the cochain determined by   $[.,.]_{\mu}-[.,.]_{\mu'}$. The map $\sigma_2:\mathbb{H} \times S\rightarrow S$   is defined as $\sigma_2(\psi,\mu)=\mu'$. The relationship between the two action $\sigma_1$ and $\sigma_2$ on $S$ is described in the following proposition.
\begin{prop}
Let $\lambda$ be a deformation of the Hom-Leibniz algebras $\LX
$ with base $\A$ and let
$$\begin{array}{ccc}
0\longrightarrow\K& \stackrel{i}{\longrightarrow}\B& \stackrel{p}{\longrightarrow}\A\longrightarrow 0 \end{array}$$
be a given extension of $\A$. If $u:\B\longrightarrow\B $ is an automorphism of this
extension which corresponds to an element $h \in H^1_{Harr}(\A,\K)=T\A$, then for any deformation $\mu$
of $\LX
$ with base $\B$, such that $p_{\ast}\mu=\lambda$, the difference $[.,. ]_{u_\mu}-[.  ,. ]_\mu$ is a cocycle in the cohomology
class $d\lambda(h)$. This means the operation $\sigma_1 $ and $\sigma_2 $ on $S$ are related to each
other by the differential $d\lambda:T\A\rightarrow \mathbb{H}.$
\end{prop}
\begin{cor}
Suppose that for a deformation $\lambda$ of the Hom-Leibniz algebra $\LX
$ with
 base $\A$, the differential $d\lambda:T\A\rightarrow \mathbb{H}$ is onto. Then the group of automorphisms $\A$ of the extension
 \eqref{(1)} operates transitively on the set of equivalence classes of deformations $\mu$ of $\LX
$
 with base $\B$ such that $p_{\ast}\mu=\lambda$. In the other words, if $\mu$ exists, it is unique up to an isomorphism
 and an automorphism of this extension.
\end{cor}
Suppose now that $M$ is finite dimensional $\A$-module satisfying 
the condition $\mathfrak{m}M=0$, where $\mathfrak{m}$ is the maximal ideal of $\A$.
The previous results can be generalized from the 1-dimensional extension to a more general extension.
$$\begin{array}{ccc}
0\longrightarrow M& \stackrel{i}{\longrightarrow}\B& \stackrel{p}{\longrightarrow}\A\longrightarrow 0. \end{array}$$
A deformation $\mu$ with base $\B$ such that $p_{\ast}\mu=\lambda$ exists if and only if the obstruction $\theta_\lambda([f]) = 0$.
 If  the differential $d\lambda:T\A\longrightarrow \mathbb{H}$ is onto, then, if $\mu$ exists, it is unique up to an isomorphism
 and an automorphism of this extension.
\begin{prop}
Suppose $\A_1$ and $\A_2$ are finite dimensional local algebras with augmentation $\varepsilon_1$ and $\varepsilon_2$,
respectively. Let $\phi:\A_2\longrightarrow\A_1$ be an algebra homomorphism with $\phi(1)=1$
and $\varepsilon_1\circ\phi=\varepsilon_2$. Suppose $\lambda_2$ is a deformation of a Hom-Leibniz algebra $\LX
$
with base $\A_2$ and $\lambda_1=\phi_\ast\lambda_2$  is the push-out via $\phi$. Then following diagram commutes.
$$
\begin{array}{ccccc} H^2_{Harr}(\A_1,\K)\\
\phi^\ast\downarrow&
\searrow^{\theta_{\lambda_1}}\\
H^2_{Harr}(\A_2,\K) & \stackrel{\theta_{\lambda_2}}{\longrightarrow}&  H
^3(\LX,\LX)\\
\end{array}
$$
\end{prop}
\section{Construction of versal deformation}
We have constructed in the class of infinitesimal deformations of a Hom-Leibniz algebra, a universal deformation, that is one which induces all the other and the homomorphism is unique. It is known that, in general,  in the category of deformations of an algebraic object there is no universal deformations. But under certain natural conditions it is possible to get a "versal" object which still induces all non-equivalent deformations. The aim of this section is to extend the construction of versal deformation, given for Leibniz algebras in \cite{Fialowoski}, to Hom-Leibniz algebras.
Let $(\LX
,[.,. ],\alpha)$ be a Hom-Leibniz algebra with finite dimensional second cohomology group ($dim(\mathbb{H})<\infty$). Consider the extension 
$$\begin{array}{ccc}
0\longrightarrow H'& \stackrel{i}{\longrightarrow}\C_1& \stackrel{p}{\longrightarrow}\C_0\longrightarrow 0 \end{array}$$
where $C_0=\K$ and $C_1=\K\oplus \mathbb{H}'$ as before. Let $\eta_1$ be the universal infinitesimal  deformation with base $C_1$.
Similarly to  classical case we proceed by induction.  Suppose for some $k\geq 1$ we have constructed
 a finite dimensional local algebra $C_k$ and a deformation $\eta_k$ of $\LX
$ with base $C_k$.
 $$ \mu: H^2_{Harr}(C_k,\K)\longrightarrow (Ch_2(C_k))'$$ be a homomorphism sending a cohomology
 class to a cocycle representing the class. Let $$ f_{C_k}: Ch_2(C_k)\longrightarrow H^2_{Harr}(\C_k,\K)' $$
  be the dual of $\mu$.  Then we have the following extension of $C_k$:
\begin{equation}\label{(11)}
\begin{array}{ccc}
0\longrightarrow H^2_{Harr}(C_k,\K)'& \stackrel{\bar{i}_{k+1}}{\longrightarrow}\bar{C}_{k+1}& \stackrel{\bar{p}_{k+1}}{\longrightarrow}C_{k}\longrightarrow 0 
\end{array}.
\end{equation}
The corresponding  obstruction $\theta_{\eta_k}([f_{C_k}])\in  H^2_{Harr}(C_k,\K)'\otimes H^3(\LX
,\LX
) $
gives  a linear map
$w_k: H^2_{Harr}(C_k,\K)\longrightarrow H^3(\LX
,\LX
) $
with the dual map
  $$w_k: H^3(\LX
,\LX
)'\longrightarrow H^2_{Harr}(C_k,\K)'. $$
We have an induced extension   $$\begin{array}{ccc}
0\longrightarrow coker(w_k')\longrightarrow\frac{\bar{C}_{k+1}}{i_{k+1}\circ w'_k(H^3(\LX,\LX))'}&\longrightarrow C_{k}\longrightarrow 0 .\end{array}$$
Since $coker(w_k')\cong(ker(w_k))'$, it yields an extension
\begin{equation}\label{extVersal1}  \begin{array}{ccc}
0\longrightarrow(ker(w_k))' & \stackrel{i_{k+1}}{\longrightarrow}C_{k+1}& \stackrel{p_{k+1}}{\longrightarrow}C_{k}\longrightarrow 0 \end{array} 
\end{equation}
where $C_{k+1}=\bar{C}_{k+1}/{i_{k+1}\circ w'_k(H^3(\LX,\LX))'}$ and $i_{k+1}$, $p_{k+1}$ are
 the mappings induced by $\bar{i}_{k+1}$, $\bar{p}_{k+1}$, respectively.
\begin{prop}
The deformation $\eta_k$ with base $C_k$ of a Hom-Leibniz algebra $\LX
$ admits an extension
to a deformation with base $C_{k+1}$, which is unique up to an isomorphism and automorphism of the extension \eqref{extVersal1}.

By induction, the above process yields a sequence of finite dimensional local
algebra $C_k$ and deformations $\eta_k$ of the Hom-Leibniz algebra $\LX
$
with base $C_k$ $$ \K\stackrel{p_{1}}{\longleftarrow} C_1\stackrel{p_{2}}{\longleftarrow}....\stackrel{p_{k}}{\longleftarrow}C_k\stackrel{p_{k+1}}{\longleftarrow}C_{k+1}...$$
  such that ${p_{k+1}}_\ast \eta_{k+1}=\eta_k$.

 Thus by taking the projective limit we obtain a formal deformation $\eta$
of $\LX
$ with base $C=\underleftarrow{lim}C_k$.
\end{prop}
\section{Massey Brackets and obstructions}
In this section we show a relationship between obstructions and Massey brackets in the case of Hom-Leibniz algebras, see  \cite{ Fial,Ashis,Retakh} for the classical case. This is needed to make more specific computation in the construction of versal deformations. The obstructions $\theta_k:H^2_{Harr}(C_k,\K)\rightarrow H^3(\LX,\LX),$ may be described in terms of Massey products.

Suppose $(C,\nu,d)$ is a differential  graded Lie algebra. The cohomology of $C$, with respect to $d$, is denoted by $\emph{H}$.  Our main example is $H(\LX,\LX)$.   Let $F$ be a graded commutative
coassociative   coalgebra, that is a graded vector space with a degree 0 mapping (comultiplication)
 $\Delta:F\rightarrow F\otimes F$ satisfying the condition $S\circ\Delta=\Delta$
and $(1\otimes\Delta)\circ\Delta=(\Delta\otimes 1)\circ\Delta$, where $S:F\otimes F\longrightarrow F\otimes F $
is defined as $S(\psi\otimes \phi)=(-1)^{|\phi||\psi|}(\phi\otimes\psi). $

Suppose also that a filtration $F_0\subset F_1\subset F$ is given in $F$ such that $F_0\subset ker(\Delta)$ and $Im(\Delta)\subset F_1\otimes F_1$.
\begin{prop}\label{prop}
Suppose a linear map $\psi: F_1\rightarrow C$ of degree 1 satisfies the condition
\begin{equation}\label{cond4}
d\psi=\nu\circ(\psi\otimes\psi)\circ\Delta.
\end{equation}
Then
$\nu\circ(\psi\otimes\psi)\circ\Delta(F)\subset ker(d).$
\end{prop}
\begin{defn}Let $a:F_0\rightarrow \emph{H}$  and $b: F/F_1\rightarrow \emph{H}$ two linear maps  of degree $1$ and $2$.
We say that $b$  is contained in the Massey $F$-bracket of $a$, and write $b\in[a]_F$,
if there exists a degree $1$ linear map $\psi:F_1\longrightarrow C$ satisfying
condition \eqref{cond4} and such that the following diagrams 
$$
\begin{array}{ccccc} F_0&\stackrel{\psi|{F_0}}{\longrightarrow}  &ker(d) \\\big\downarrow{id}
&&\big\downarrow \pi
\\F_0& \stackrel{a}{\longrightarrow}&\emph{H}\\
\end{array}\;\;\;\qquad
\begin{array}{ccccc} F&\stackrel{\nu\circ(\psi\otimes\psi)\circ\Delta}{\longrightarrow}  &ker(d) \\\big\downarrow{\pi}
&&\big\downarrow \pi\\ \frac{F}{F_1} & \stackrel{b}{\longrightarrow}&\emph{H}\\
\end{array}
$$
are commutative, where $\pi$ denotes the projection of each space onto the quotient space.
\end{defn}
Note that the upper horizontal maps of the above diagrams are well defined,
since $\psi(F_0)\subset \psi(ker\Delta) \subset ker(d)$,
and $\nu\circ(\alpha\otimes\alpha)\circ\Delta(F)\subset ker(d)$)
by Proposition \ref{prop}. 

The definition makes sense even if $F_1 = F$. In that case $Hom(F/F_1,\K)=0$, and $[a]_F$
may either be empty or contain 0. In that case we say that
a satisfies the condition of triviality of Massey $F$-brackets.

We consider the differential graded Lie algebra $( \mathcal{C} ^\ast(\LX,\LX), \nu, d)$.
 Let $F = F_1 =\mathfrak{m}'$,
 the dual of $\mathfrak{m}$ and $F_0 = (\mathfrak{m}/\mathfrak{m}^2)$. Let $\Delta: F\longrightarrow F\otimes F$
  be the comultiplication in $F$ which is the dual of the multiplication in $\mathfrak{m}$. Then $F$ is a
cocommutative coassociative coalgebra. 

For a linear functional $\phi: \mathfrak{m}\rightarrow \K$
define a map $\psi_\phi: \LX \otimes \LX \rightarrow\LX
$ by
$$\psi_\phi(l_1, l_2) = ( \phi\otimes id)([1 \otimes l_1, 1 \otimes l_2]_\lambda - 1 \otimes[l_1, l_2]).$$
This gives $\psi: \mathfrak{m} \rightarrow  \mathcal{C} ^2(\LX,\LX)$ by $\phi\rightarrow\psi_\phi$. From the definition it is clear that
 $[,]_\lambda $ and $\psi $ determine each other. Then we have
\begin{prop} The operation $[.,. ]_\lambda $ satisfies the Hom-Leibniz identity if and only if
$\psi$ satisfies the equation $d\psi-\frac{1}{2}\nu\circ(\psi\otimes\psi)\circ\Delta=0$.
\end{prop}
\begin{proof}
Let$\{m_i\}$ be a basis of $\mathfrak{m}$. We can write
$$ [1\otimes l_1,1\otimes l_2]_\lambda=1\otimes[l_1,l_2]+\sum_{i=1}^r m_i\otimes \psi_i(l_1,l_2)\; \hbox{for}\; l_1,l_2\in \LX,
$$
where  $\psi_i\in  \mathcal{C} ^2(\LX,\LX)$ is given by  $\psi_i = \psi_\mathfrak{m_i'}$.
Thus 
\begin{eqnarray*}
[1\otimes \alpha(l_1), [1\otimes l_2, 1\otimes l_3]_\lambda]_\lambda&=&\big[1\otimes \alpha(l_1), 1\otimes[l_2, l_3]+\sum_{i=1}^r m_i\psi_i(l_2, l_3)\big]_\lambda \\ &=& [1 \otimes \alpha(l_1), 1 \otimes[l_2, l_3]]_\lambda +\sum_i
m_i [1\otimes \alpha(l_1), 1 \otimes \psi_i(l_2, l_3)]_\lambda
\\ &=& 1\otimes[\alpha(l_1), [l_2, l_3]] +\sum_im_i\otimes\psi_i(\alpha(l_1), [l_2, l_3])\\ && +\sum_i m_i \otimes[\alpha(l_1),\psi_i(l_2, l_3)]
+\sum_{i,j}m_im_j\otimes\psi_j(\alpha(l_1), \psi_i(l_2, l_3)).
\end{eqnarray*}
Similarly  
\begin{eqnarray*}\big[[1\otimes l_1, 1\otimes l_2]_\lambda, 1\alpha(l_3)\big]_\lambda
&=& 1 \otimes[[l_1, l_2], \alpha(l_3)] +\sum_i
m_i\otimes \psi_i([l_1, l_2], \alpha(l_3))\\ & & +\sum_i
m_i\otimes [ \psi_i(l_1, l_2), \alpha(l_3)]\\ & &+\sum_{i,j}m_im_j \otimes\psi_j ( \psi_i(l_1, l_2), \alpha(l_3)),
\end{eqnarray*}
and
\begin{eqnarray*}
\big[[1\otimes l_1, 1\otimes l_3]_\lambda, 1\alpha(l_2)\big]_\lambda&= &1 \otimes[[l_1, l_3], \alpha(l_2)] +\sum_i
m_i \otimes\psi_i([l_1, l_3], \alpha(l_2)) \\ & &+\sum_i
m_i [ \psi_i(l_1, l_3), \alpha(l_2)]+\sum_{i,j}m_im_j \otimes\psi_j ( \psi_i(l_1, l_3), \alpha(l_2)).
\end{eqnarray*}
Let $\Delta(\phi)=\sum_p\xi_p\otimes\gamma_p$ for some $ \xi_p,\gamma_p\in \mathfrak{m}'$. 
We set
$\xi_p(m_i)=\xi_{p,i}$ and $\gamma_p(m_i)=\gamma_{p,i}$.
Thus 
$$\phi(m_im_j)=\Delta(\phi)(m_i\otimes m_j)=\sum_p(\xi_p\otimes\gamma_p)(m_i\otimes m_j)=\sum_p\xi_{p,i}\gamma_{p,j},$$
and 
\begin{eqnarray*}
(\phi\otimes id)\sum_{i,j}m_im_j\otimes \psi_j(\alpha(l_1), \psi_i(l_2, l_3))
 = \sum_{i,j,p}\xi_{p,i}\gamma_{p,j}\psi_j(\alpha(l_1), \psi_i(l_2, l_3))
\\=\sum_p \sum_j\gamma_{p,j}\psi_j(\alpha(l_1),\sum_i \xi_{p,i}\psi_i(l_2, l_3))
=\sum_p \psi_{\gamma_p}(\alpha(l_1),\psi_{\xi_p}(l_2, l_3)).
\end{eqnarray*}
Therefore 
\begin{eqnarray*}
&&(\phi\otimes id)([1\otimes \alpha(l_1), [1\otimes l_2, 1\otimes l_3]_\lambda]_\lambda)
\\&& =\sum_i\phi(m_i)\otimes \psi_i(\alpha(l_1), [l_2, l_3]) +\sum_i\phi(m_i)\otimes [\alpha(l_1),\psi_i(l_2, l_3)] Ê+\sum_p \psi_{\gamma_p}(\alpha(l_1),\psi_{\xi_p}(l_2, l_3))\\
&&=\psi_\phi(\alpha(l_1),[l_2,l_3])+[\alpha(l_1),\psi_\phi(l_2,l_3)]+\sum_p \psi_{\gamma_p}(\alpha(l_1),\psi_{\xi_p}(l_2, l_3)).
\end{eqnarray*}
Similarly, we calculate
$(\phi\otimes id)\big[[1\otimes l_1, 1\otimes l_2]_\lambda, 1 \otimes \alpha(l_3)\big]_\lambda$ and
$(\phi\otimes id)\big[[1\otimes l_1, 1\otimes l_3]_\lambda, 1 \otimes \alpha(l_2)\big]_\lambda$.
We get
 $$(\phi\otimes id)[1\otimes \alpha(l_1), [1\otimes l_2, 1\otimes l_3]_\lambda]_\lambda-  \big[[1\otimes l_1, 1\otimes l_2]_\lambda, 1\otimes\alpha(l_3)\big]_\lambda+\big[[1\otimes l_1, 1\otimes l_3]_\lambda, 1\otimes\alpha(l_2)\big]_\lambda$$
$$=(-d\psi+\frac{1}{2}\nu\circ(\psi\otimes\psi)\circ\Delta)\phi(l_1,l_2,l_3). $$
Thus, it follows that $[.,.]_\lambda$ satisfies the Hom-Leibniz identity if and only if $\psi$ satisfies the equation $d\psi-\frac{1}{2}\nu\circ(\psi\otimes\psi)\circ\Delta=0.$
\end{proof}
\begin{cor} A linear map $a : F_0\rightarrow H$ is a differential of some deformation
with base $\A$ if and only if $\frac{1}{2}a$ satisfies the condition of triviality of Massey F-brackets.
\end{cor}
\begin{thm} The obstruction $ \theta_\lambda$ has the property, $2w_k= [id]_F$ . Moreover,
an arbitrary element of $[id]_F$ is equal to $2w_k$ for an appropriate extension of the
deformation $\eta_1$, of $\LX
$ with base $C_1$,  to a deformation $\eta_k$ of $\LX
$ with base $C_k$.
\end{thm}
\begin{proof}
As above we define the maps $$\psi_k:\mathfrak{m}_k'\longrightarrow  \mathcal{C} ^2(\LX,\LX)$$
by $\psi_\phi(l_1,l_2)=(\phi\otimes id)\big([1\otimes l_1,1\otimes l_2]_{\eta_k}-1\otimes[l_1,l_2]\big)$
for $\phi \in \mathfrak{m}_k'$ and $l_1,l_2\in \LX
$, using the deformation  $\eta_k$ with base $C_k$.
Since $\eta_k$ is a Hom-Leibniz algebra structure on $C_k\otimes \LX
$,
Proposition \ref{prop} implies $$d\psi=\frac{1}{2}\nu\circ(\psi\otimes\psi)\circ\Delta. $$
Different $\psi$ with these properties
correspond to different  extensions $\eta_k$ of $\eta_1$.

The ${\psi_k}_{|F_0}$ is given by $\psi_k(h_i)=\mu(h_i)$,
a representative of the cohomology class $h_i$.
So  $a=\pi\circ{ {\psi_k}_{|F_0}}=id$.
In the definition of Massey $F$-bracket, the map
 $b: \frac{F}{F_1}\rightarrow H^3(\LX,\LX)$.
If we consider $\{m_i\}_{1\leq i\leq r}$  a basis  of $\mathfrak{m}_k$
and extend it to a basis $\{\bar{m}_i\}_{1\leq i\leq r+s}$  of
$\mathfrak{\bar{m}}_k$, the bracket is  given by
$$ [1\otimes l_1,1\otimes l_2]_{\eta_k}=1\otimes [l_1,l_2]+\sum_{i=1}^{r}\bar{m}_i\otimes\psi_i(l_1,l_2)$$
for  arbitrary cochains $\psi_i \in  \mathcal{C} ^2(\LX,\LX)$ for $r\leq i\leq s$
the $\bar{C}_{k+1}$-bilinear map  $\{ . ,. \}$
on $\bar{C}_{k+1}\otimes \LX
$ is given by
\begin{eqnarray*}
\{1\otimes l_1,1\otimes l_2\}&=&1\otimes [l_1,l_2]+\sum_{i=1}^{r+s}\bar{m}_i\otimes\psi_i(l_1,l_2)\{\{1\otimes l_1,1\otimes l_2\},1\otimes \alpha(l_3)\}\\
& =& 1 \otimes[[l_1, l_2], \alpha(l_3)] +\sum_{i=1} ^{r+s}\bar{m}_i\otimes \psi_i([l_1, l_2], \alpha(l_3)) \\
& &+\sum_{i=1}^{r+s} \bar{m}_i\otimes [ \psi_i(l_1, l_2), \alpha(l_3)] +\sum_{i,j=1}^{r+s}\bar{m}_i\bar{m}_j \otimes\psi_j ( \psi_i(l_1, l_2),\alpha(l_3))\\
&=& 1 \otimes[[l_1, l_2], \alpha(l_3)] +\sum_{i=1} ^{r+s}\bar{m}_i\otimes \psi_i([l_1, l_2], \alpha(l_3)) \\
& & +\sum_{i=1}^{r+s} \bar{m}_i\otimes [ \psi_i(l_1, l_2), \alpha(l_3)]+\sum_{i,j=1}^{r+s}\sum_{p=1}^{r}c_{ij}^p\bar{m}_p \otimes\psi_j ( \psi_i(l_1, l_2),\alpha(l_3)).
\end{eqnarray*}
Similarly, we calculate $\{1\otimes \alpha(l_1),\{1\otimes l_2,1\otimes l_3\}\}$
and $\{\{1\otimes l_1,1\otimes l_3\},1\otimes \alpha(l_2)\}.$

Therefore 
\begin{eqnarray*}
&&(\bar{m}'_p\otimes id)(\{1\otimes \alpha(l_1),\{1\otimes l_2,1\otimes l_3\}\}- \{\{1\otimes l_1,1\otimes l_2\},1\otimes \alpha(l_3)\}\\&& +\{\{1\otimes l_1,1\otimes l_3\},1\otimes \alpha(l_2)\}) \\&& 
=d\psi_p(l_1,l_2,l_3)+\frac{1}{2} \nu\circ(\psi\otimes\psi)\circ\Delta(\bar{m}'_p)(l_1,l_2,l_3).
\end{eqnarray*}
The result follows by taking $b=2w_k$ and $a=id_{|H}$.
\end{proof}

\section{Examples}
 Let  $\LX
$ be a 3-dimensional vector
 space with basis $\{e_1,e_2,e_3\}$. Define a bracket
$[.,.]:\LX
\otimes \LX
 \rightarrow \LX
$ by
\begin{equation}\label{CrochetExample}
[e_1, e_3] = e_2,\; \;[e_3, e_3] = e_1,
\end{equation}
and, all other brackets of basis elements being zero.   The triple $(\LX , [. , . ], \alpha )$ defines a Hom-Leibniz algebra if $\alpha$ is, with respect to the basis $\{ e_1,e_2,e_3\}$, of the form
\begin{equation*}
\begin{cases}
 \alpha (e_1)=a_{11}e_1+a_{21} e_2,\\
 \alpha (e_2)=a_{12}e_1+a_{22} e_2,\\
 \alpha (e_3)=a_{13}e_1+a_{23} e_2+a_{33}e_3,
\end{cases}
\end{equation*}
where $a,\ b$ are arbitrary parameters.
The linear map $\alpha$ is multiplicative if it reduces to
\begin{equation}\label{MultiAlphaExamp}
\begin{cases}
 \alpha (e_1)=c^2 e_1+a c e_2,\\
 \alpha (e_2)=c^3 e_2,\\
 \alpha (e_3)=a e_1+b  e_2+ce_3,
\end{cases}
\end{equation}
where $a,b,c$ are arbitrary parameters.

To construct a versal deformation of  the multiplicative Hom-Leibniz algebra $(\LX
,[.,.  ],\alpha)$,
we need first to calculate cohomology spaces. We  consider two different cases for $\alpha$.

\begin{example}[$c\neq1,c\neq0$ and arbitrary $a,b$]
Let $(\LX , [\ , \ ], \alpha )$ be a Hom-Leibniz algebra, where the bracket is defined in \ref{CrochetExample} and the map $\alpha$ as
\begin{equation}\label{MultiAlphaExamp}
\begin{cases}
 \alpha (e_1)= e_1+a  e_2,\\
 \alpha (e_2)= e_2,\\
 \alpha (e_3)=a e_1+b  e_2+e_3,
\end{cases}
\end{equation}
where $a,b$ are arbitrary parameters.
Let $\phi\in Z^2(\LX,\LX)$ be a 2-cocycle. Then $\phi:\LX
\otimes\LX
\rightarrow\LX
$
is a linear map satisfying $\delta\phi(e_i,e_j,e_k)=0$ and $\phi(\alpha(e_i),\alpha(e_j))=\alpha(\phi(e_i,e_j))$,
 for $ 1\leq i,j,k\leq3.$

The matrix $M_\phi$ of $\phi$, with respect to the ordered basis $B=\{e_i\otimes e_j\}_{1\leq i,j\leq3}$ of $\LX\otimes \LX
$, is given  $\forall a,b \in \R$ and $c=1$, by

$$\delta\phi=0  \longleftrightarrow
 M_\phi=\left(
                                                \begin{array}{ccccccccc}
                                                  0 & 0 & x_3 & 0 & 0& x_5 & x_1 & x_2 & x_7 \\
                                                  x_1 & x_2 & x_4 & 0 & 0 & x_6 & 0 & ax_2 & x_8 \\
                                                  0 & 0 & -x_2& 0 & 0 & 0  & 0 & 0 & -x_1 \\
                                                \end{array}
                                              \right)
$$

If in addition we have the condition
$\phi(\alpha(e_i),\alpha(e_j))=\alpha(\phi(e_i,e_j))$, we obtain
$$\left\{
     \begin{array}{ll}
       \delta\phi=0,\; \\
and \hskip 3cm\longleftrightarrow\\
       \phi(\alpha(e_i),\alpha(e_j))=\alpha(\phi(e_i,e_j))
     \end{array}
   \right.
   M_\phi=\left(
                                                                                                \begin{array}{ccccccccc}
                                                                                                  0 & 0 & 0 & 0 & 0 & 0 & 0 & 0 & 0 \\
                                                                                                  0 & 0 & 0 & 0 & 0 & 0 & 0 & 0 & 1 \\
                                                                                                  0 & 0 & 0 & 0 & 0 & 0 & 0 & 0 & 0 \\
                                                                                                \end{array}
                                                                                            \right)$$

Hence $Z^2(\LX,\LX)=\langle\left(
                                                                                                \begin{array}{ccccccccc}
                                                                                                  0 & 0 & 0 & 0 & 0 & 0 & 0 & 0 & 0 \\
                                                                                                  0 & 0 & 0 & 0 & 0 & 0 & 0 & 0 & 1 \\
                                                                                                  0 & 0 & 0 & 0 & 0 & 0 & 0 & 0 & 0 \\
                                                                                                \end{array}
                                                                                              \right)\rangle =\langle E_{29}\rangle$
\\ Let  $\phi_0\in B^2(\LX,\LX)$. Then there is a 1-cochain $f\in \mathcal{C}^1(\LX,\LX)=Hom(\LX,\LX)$ such that $\phi_0=\delta f$.
Let $f(e_i) =x_ie_1 + y_ie_2 + z_3e_3$ for $i = 1, 2, 3$.
We have
$$\delta f(e_i, e_j) = [e_i, f(e_j)] + [f(e_i), e_j ] -f([e_i, e_j ]).$$
Then  the matrix of $\delta f$  can be written as
$$\delta f=\left(
             \begin{array}{ccccccccc}
               0 & 0 &z_1-x_2  & 0 & 0 & z_2 &z_1 & z_2 &2z_3-x_1\\
               z_1 & z_2 &z_3+x_1-y_2  & 0 & 0 & x_2 & 0 & 0 &x_3-y_1  \\
               0 & 0 & -z_2 & 0 & 0 & 0 & 0 & 0 &    -z_1\\
             \end{array}
           \right)$$
Since $\phi_0=\delta f$ and  $\phi_0(\alpha(e_i),\alpha(e_j))=\alpha_0(\phi(e_i,e_j)),$
it turns out  that:
\begin{equation*}
\phi_0= x E_{29}
\end{equation*}

\begin{equation*}
    B^2(\LX,\LX)=\langle E_{29}\rangle
\end{equation*}
It follows that
 if $c\neq1,c\neq0$ and $\forall a,b\in \R$ we have

  \begin{equation*}
  H^2(\LX
,\LX
)=\{0\}.\end{equation*}

Hence, every formal deformation  is equivalent to
a trivial deformation.

\end{example}
\begin{example}[$c=0, \;a\neq 0$ and  $b$ arbitrary]
Let $(\LX , [. , . ], \alpha )$ be a Hom-Leibniz algebra, where the bracket is defined in \ref{CrochetExample} and the map $\alpha$ as
\begin{equation}\label{MultiAlphaExamp}
\begin{cases}
\alpha (e_1)=0,\\
 \alpha (e_2)=0,\\
 \alpha (e_3)=a e_1+b  e_2,
\end{cases}
\end{equation}
where $a,b$ are arbitrary parameters.

In this case we have

$\; \delta^2\phi=0  \longleftrightarrow$ $$ M_\phi=\left(
                                             \begin{array}{ccccccccc}
                                               \frac{-b}{a}x_{14} &\frac{-b}{a}x_{15}  & x_{13} & x_{14} &x_{15}&x_{16} & x_{17} & x_{18} &
                                               x_{19} \\
                                               x_{21} & 0 & x_{23}&x_{24} &x_{25} &x_{26} &x_{27} & x_{28} & x_{29} \\
                                               0 & 0 & \frac{-b}{a}x_{25} & 0 & 0 & 0 & 0 & 0 &-ax_{21}-bx_{24}\\
                                             \end{array}
                                           \right)
$$

$$\left\{
  \begin{array}{ll}
 \delta^2\phi=0\\
and \hskip 3 cm \longleftrightarrow \\
    \phi(\alpha(e_i),\alpha(e_j))=\alpha(\phi(e_i,e_j))
  \end{array}
\right.
M_\phi=\left(
                                                                                                 \begin{array}{ccccccccc}
                                                                                                   0 & 0& x_1 & 0 &0 & x_4 & x_6 & x_7 & x_8 \\
                                                                                                   0 & 0& x_2 & 0 & 0 & x_5& x_9 & x_{10} & x_{11} \\
                                                                                                   0 &0 &x_3 & 0 & 0 & 0 & 0 & 0 & 0 \\
                                                                                                 \end{array}
                                                                                               \right)
$$
Hence

 \begin{equation*}
 Z^2(\LX,\LX)=\langle E_{13},E_{16},E_{17},E_{18},E_{19},E_{23},E_{26},E_{27},E_{28},E_{29},E_{33}\rangle .
\end{equation*}

By direct calculations we obtain
$$H^2(\LX,\LX)=\langle E_{16},E_{27},E_{28},E_{13},E_{33},E_{17},E_{18}\rangle=\langle\mu_1,\mu_2,\mu_3,\mu_4,\mu_5,\mu_6,\mu_7\rangle .$$
Since we are dealing with infinitesimal deformations, it turns out that we obtain
$$[1\otimes e_i,1\otimes e_j]_{\eta_1}=1\otimes[e_i,e_j]+\sum_{k=1}^7t_k\otimes\mu_k,$$ where $t^k$ corresponds to the dual of $\mu_k$.

\end{example}
\begin{example}
Now, we consider new brackets obtained using twisting principle. Set
\begin{align*}& [e_1, e_3] = e_2,\\
& [e_3, e_3] = e_1+a e_2,
\end{align*}
and $\alpha$ defined as
\begin{align*}
& \alpha (e_1)=e_1+a e_2,\\
& \alpha (e_2)=e_2,\\
& \alpha (e_3)=a e_1+b e_2+e_3,
\end{align*}
where $a,\ b$ are arbitrary parameters.

$\; \delta^2\phi=0  \longleftrightarrow$
\begin{equation}\label{MATR}
    \left(
      \begin{array}{ccccccccc}
        0 & 0 & x & 0 & 0 & z & 0 & 0 & v \\
         0& 0 & y& 0& 0 & t& -au & u & w\\
        0 & 0 & 0 & 0 & 0 & 0 & 0 & 0 & 0 \\
      \end{array}
    \right)
\end{equation}
since $\alpha$ respect the bracket  we obtain  the  space of 2-Hom-cocycles
$$Z^2(\LX,\LX)=\langle E_{23},E_{26},E_{29} \rangle$$
 the  space  of 
2-Hom-coboundaries  
$B^2(\LX,\LX)=\langle E_{23},E_{29} \rangle$

Then the  second cohomology group of the Hom-Leibniz algebra $(\LX, [\ , \ ],\alpha)$ is one-dimensional and  is given by:
\begin{equation*}
    H^2(\LX,\LX)
    =\langle E_{26} \rangle.
\end{equation*}
\end{example}


\begin{thebibliography}{99}
\bibitem {Ammar} Ammar F., Ejbehi Z.  and Makhlouf A., \emph{Cohomology and Deformations of Hom-algebras}, Journal of Lie Theory\textbf{ 21}, No. 4 (2011), 813--836.

\bibitem {F.Ammar} Ammar F. and Makhlouf A., \emph{Hom-Lie algebras and Hom-Lie admissible superalgebras}, Journal of Algebra, \textbf{324} (7)   (2010), 1513--1528.

\bibitem{Harrison} Barr M., \emph{ Harrison homology, Hochschild homology and triples}," Journal  of Algebra, \textbf{8} (1963), 314--323.


\bibitem{right} Dzhumadil'Daev A., \emph{Cohomology and deformations of right symmetric algebras}, Journal of Math. Sciences
Volume\textbf{ 93}, Number 6, (1998) , 836--876.

\bibitem{F1}
 Fialowski A., \emph{Deformations of Lie algebras}, Mat. Sbornyik
USSR, 127 (169), (1985), 476--482; English translation: Math.
USSR-Sb., \textbf{55}, no. 2 , (1986), 467--473.

\bibitem{F2} Fialowski A.,
\emph{An example of formal deformations of Lie
algebras}, NATO Conference on Deformation Theory of Algebras and
Applications, Il Ciocco, Italy, 1986, Proceedings. Kluwer,
Dordrecht, 1988, pp. 375--401.

\bibitem{Fial}Fialowski A. and  Fuchs D., \emph{Construction of versal deformation of Lie algebra}, Journal of Functional Analysis \textbf{161} (1999), 76--110.

\bibitem{Fialowoski}Fialowski A., Mandal A. and  Mukherjee G., \emph{Versal deformation of Leibniz algebra},  J. of K-Theory, doi:10.1017/is008004027jkt049 (2008).

\bibitem{Extensions} Fialowski A. and  Penkava M., \emph{Extensions of (super) Lie algebra}, Comm. in Contemp. Math., \textbf{11} (2009), 709--737.

\bibitem{fialowski-operad} Fialowski A.,  Mukherjee G.,  Naolekar A., \emph{ Versal deformation theory of algebras over a quadratc operad},  	arXiv:1202.2967 [math.KT] (2012).
 
\bibitem{Fuks} Fuks D.B., \emph{Cohomology of infinite-dimensional Lie algebras}, Plenum, New York, 1986.

\bibitem{FuksLang} Fuks D.B. and Lang L., \emph{Massey products and Deformations}, J. of Pure and Applied Algebra, \textbf{156}, Issues 2Ð3, (2001),  215--229.

\bibitem{Gerst coho} Gerstenhaber M., \emph{The cohomology structure of an associative ring},  Ann of Math., \textbf{78} (2) (1963), 267--288.

\bibitem{Gerst def}  Gerstenhaber M., \emph{On the deformation of rings and algebras}, Ann of Math. \textbf{79} (1) (1964), 59--108.

\bibitem{Harrison0} Harrison D.K., \emph{ Commutative algebras and cohomology}, Trans. Amer. Math. Soc., \textbf{104} (1962), 191--201.

 \bibitem{Harwig Silv} Hartwig J.T., Larsson D.  and Silvestrov S.D., \emph{Deformations of Lie algebras using $\sigma-$derivations}, Journal of  Algebra \textbf{295} (2006),  314-361.
 
\bibitem{semi s}Jin Q.  and Li   X., \emph{Hom-Lie algebra structures on semi simple Lie algebras}, Journal of Algebra \textbf{319} (2008), 1398--1408.

\bibitem{Kontsevish} Kontsevich M.  and  Soibelman Y.
\emph{Deformations of algebras over operads and the Deligne conjecture},
Conf«erence Mosh«e Flato 1999, Vol. I , 255Ð307, Math. Phys. Stud., 21, Kluwer Acad.
Publ., Dordrecht, 2000.

\bibitem{LS1} Larsson, D., Silvestrov, S.D.,
Quasi-hom-Lie algebras, Central Extensions and 2-cocycle-like
identities, J. Algebra \textbf{288} (2005), 321--344.


\bibitem{Laudal} Laudal O.A., \emph{Formal moduli of algebraic structures}, Lect. Notes in Math. \textbf{754}, Springer-Verlag 1979.


\bibitem{Lecomte}  Lecomte P.A.B. and   Schicketanz H., \emph{The multigraded Nijenhuis-Richardson algebra, its universal property and applications}, arxiv: math/920T257vT[matj.QA]

\bibitem{Loday}  Loday J-L., \emph{Une version non commutative des alg\`ebres de Lie: les alg\`bres
de Leibniz,} Enseign. Math., \textbf{39} (2), No. 3-4 (1993), 269--293.

\bibitem{Loday2} Loday J-L., \emph{Overview on Leibniz algebras, dialgebras and their homology}, Field Institute Communications, \textbf{17}  (1997) 91--102.

\bibitem{LodayPira} Loday J-L. and Pirashvili T., \emph{Universal enveloping algebras of Leibniz algebras and (co)homology}, Math. Ann., \textbf{296} (1993) 139--158.

 \bibitem{Makhl Silv Hom}  Makhlouf A. and Silvestrov S., \emph{Hom-algebra structures}, J. Gen. Lie Theory Appl. Vol \textbf{2} (2),  (2008), 51--64.
 

\bibitem{HomAlgHomCoalg} Makhlouf A. and Silvestrov S., 
\emph{Hom-Algebras and Hom-Coalgebras}, Journal of Algebra and its Applications, Vol. \textbf{9}   (4),  (2010), 553--589.

 \bibitem{Makhl Silv Not} Makhlouf A. and  Silvestrov S.D., \emph{Notes on formal deformations of Hom-associative and Hom-Lie algebras}, Forum Mathematicum, vol. \textbf{22} (4) (2010), 715--739.

   \bibitem{HomNonAss} Makhlouf A., \emph{Paradigm of Nonassociative Hom-algebras and Hom-superalgebras},  \emph{Proceedings of Jordan Structures in Algebra and Analysis Meeting}, Eds:  J. Carmona Tapia, A. Morales Campoy, A. M. Peralta Pereira, M. I. Ramírez Álvarez,
Publishing house: Circulo Rojo, ( 145--177).


\bibitem {Ashis}Mandal A., \emph{An Example of Constructing Versal Deformation for Leibniz Algebras}, arXiv:0712.2096v1 (2007).

\bibitem{MerkulovVallette} Merkulov S. and  Vallette B.,  \emph{Deformation theory of representations of prop(erad)s.} I. J. Reine Angew. Math. \textbf{634} (2009), 51--106.


\bibitem{NijenhuisRichardson} Nijenhuis A. and Richardson R., \emph{Deformation of Lie algebras structures}, Jour. Math. Mech. \textbf{17} (1967), 89--105.


\bibitem{Retakh} Retakh V.S., \emph{The Massey operations in Lie superalgebras and deformations of Complexly Analytic algebras}, Funktsional Anal. i Prilozhen.  (1977), no.4, 88; English transl. in Functional Anal. Appl.11 (1977).

 \bibitem{Rotki} Rotkiewicz  M., \emph{Cohomology  Ring of $n$-Lie Algebras}, Extracta Mathematicae vol. \textbf{20} (3) (2005), 219--232.
 
 \bibitem{Sheng} Sheng Y., \emph{Representations of hom-Lie algebras},  Algebra and Representation Theory,  15 (6) (2012), 1081-1098. DOI:10.1007/s10468-011-9280-8.

\bibitem{Schlesinger} Schlesinger M., \emph{Functors of Artin rings,} Trans. Amer. Math. Soc.\textbf{130} (1968), 208-222.

 \bibitem{stasheff} Stasheff J.D., \emph{The intrinsic bracket on the deformation complex of an associative algebra}, Journal of Pure and Applied Algebra \textbf{89} (1993), 231--235.

\bibitem{Yau:EnvLieAlg} Yau D.,
\emph{ Enveloping algebra of Hom-Lie algebras,}  J. Gen.
Lie Theory Appl. \textbf{2} (2) (2008), 95--108.

  \bibitem{Yau homol} Yau D., \emph{Hom-algebras and homology},  J. Lie Theory 19 (2009) 409-421.




\bibitem{Yau:YangBaxter2} Yau D., 
\emph{The Hom-Yang-Baxter equation and Hom-Lie algebras},  J. Math. Phys. \textbf{52} (2011), 053502.

 
\end{thebibliography}
\end{document}